\documentclass[10pt,oneside,a4paper]{amsart}
\usepackage{amsmath,amsfonts,amsthm, amssymb}
\usepackage{array}
\usepackage[all,textures]{xy}
\usepackage{portland}

\theoremstyle{plain}
\newtheorem{theorem}{Theorem}[section]
\newtheorem{proposition}[theorem]{Proposition}
\newtheorem{lemma}[theorem]{Lemma}
\newtheorem{corollary}[theorem]{Corollary}

\theoremstyle{definition}
\newtheorem{definition}[theorem]{Definition}

\theoremstyle{remark}

\newtheorem{remark}[theorem]{Remark}
\newtheorem{example}[theorem]{Example}

\newcommand{\ovl}{\overline}
\newcommand{\Kern}{\mathrm{Ker}}

\newcommand{\cone}{\mathrm{cone}}

\renewcommand{\lim}{\mathrm{lim}}

\newcommand{\Spec}{\mathrm{Spec}}

\newcommand{\Ext}{\mathrm{Ext}}

\newcommand{\Hom}{\mathrm{Hom}}

\newcommand{\RHom}{\mathrm{RHom}}

\newcommand{\op}{^{\mathrm{op}}}
\newcommand{\Ob}{\mathrm{Ob}}
\newcommand{\Z}{\mathbb{Z}}

\newcommand{\NNN}{\mathfrak{n}}

\newcommand{\PPP}{\mathfrak{p}}

\newcommand{\GGGG}{\mathfrak{g}}

\newcommand{\CC}{\mathbf{C}}

\newcommand{\Mod}{\ensuremath{\mathsf{Mod}} }

\newcommand{\mmod}{\ensuremath{\mathsf{mod}} }

\newcommand{\Rng}{\ensuremath{\mathsf{Rng}} }

\newcommand{\Qch}{\ensuremath{\mathsf{Qch}} }

\newcommand{\Cat}{\ensuremath{\mathsf{Cat}} }

\newcommand{\Ind}{\ensuremath{\mathsf{Ind}}}
\newcommand{\Pro}{\ensuremath{\mathsf{Pro}}}

\newcommand{\Des}{\ensuremath{\mathrm{Des}}}

\newcommand{\lra}{\longrightarrow}

\newcommand{\aaa}{\ensuremath{\mathcal{A}}}

\newcommand{\ccc}{\ensuremath{\mathcal{C}}}
\newcommand{\ddd}{\ensuremath{\mathcal{D}}}
\newcommand{\eee}{\ensuremath{\mathcal{E}}}

\newcommand{\iii}{\ensuremath{\mathcal{I}}}

\newcommand{\LLL}{\ensuremath{\mathcal{L}}}

\newcommand{\ooo}{\ensuremath{\mathcal{O}}}

\newcommand{\sss}{\ensuremath{\mathcal{S}}}
\newcommand{\ttt}{\ensuremath{\mathcal{T}}}

\CompileMatrices
\SelectTips{cm}{11}

\title{On compact generation of deformed schemes}
\author{Wendy Lowen} 
\address[Wendy Lowen]{Departement Wiskunde-Informatica, Universiteit Antwerpen, Middelheimcampus,
Middelheimlaan 1,
2020 Antwerp, Belgium}
\email{wendy.lowen@ua.ac.be}
\author{Michel Van den Bergh} 
\address[Michel Van den Bergh]{Departement WNI, Universiteit Hasselt,
3590 Diepenbeek, Belgium}
\email{michel.vandenbergh@uhasselt.be}
\thanks{The first author acknowledges the support of the European Union for the ERC grant No 257004-HHNcdMir}
\thanks{The second author is a director of research at the Fund for Scientific Research, Flanders}

\begin{document}
\maketitle

\begin{abstract}
We obtain a theorem which allows to prove compact generation of derived categories of Grothendieck categories, based upon certain coverings by localizations. This theorem follows from an application of Rouquier's cocovering theorem in the triangulated context, and it implies Neeman's result on compact generation of quasi-compact separated schemes. We prove an application of our theorem to non-commutative deformations of such schemes, based upon a change from Koszul complexes to Chevalley-Eilenberg complexes.
\end{abstract}

\section{Introduction}

Compact generation of triangulated categories was introduced by Neeman in \cite{neeman2}. One of the motivating situations is given by derived categories of ``nice'' schemes (i.e. quasi-compact separated schemes in \cite{neeman2}, later extended to quasi-compact quasi-separated schemes by Bondal and Van den Bergh 
in \cite{bondalvandenbergh}). The ideas of the proofs later cristalized in Rouquier's (co)covering theorem \cite{rouquier} which describes a certain covering-by-Bousfield-localizations situation in which compact generation (later extended to $\alpha$-compact generation 
by Murfet
in \cite{murfet2}) of a number of ``smaller pieces'' entails compact generation of the whole triangulated category. The notions needed in the (co)covering concept can be interpreted as categorical versions of standard scheme constructions like unions and intersections of open subsets, and in the setup of Grothendieck categories rather than triangulated categories they have been important in non-commutative algebraic geometry (see eg. \cite{vanoystaeyen}, \cite{vanoystaeyenverschoren}, \cite{rosenberg}, \cite{smith}). In this paper we apply Rouquier's theorem in order to obtain a (co)covering theorem for Grothendieck categories based upon these notions, which can be used to prove compact generation of derived categories of Grothendieck categories (see Theorem \ref{cocovergroth} in the paper). 

\begin{theorem}\label{cocovergrothintro}
Let $\ccc$ be a Grothendieck category with a compatible covering of affine localizing subcategories $\sss_i \subseteq \ccc$ for $i \in I = \{1, \dots, n\}$. Suppose:
\begin{enumerate}
\item $D(\ccc/\sss_i)$ is compactly generated for every $i \in I$.
\item For every $i \in I$ and $\varnothing \neq J \subseteq I \setminus \{i\}$, suppose the essential image $\eee$  of
$$\cap_{j \in J} \sss_j \lra \ccc \lra \ccc/ \sss_i$$
is such that $D_{\eee}(\ccc/\sss_i)$ is compactly generated in $D(\ccc/ \sss_i)$.
\end{enumerate}
Then $D(\ccc)$ is compactly generated, and an object in $D(\ccc)$ is compact if and only if its image in every $D(\ccc/ \sss_i)$ for $i \in I$ is compact.
\end{theorem}

When applied to the category of quasi-coherent sheaves over a quasi-compact separated scheme, the theorem implies Neeman's original result.

Our interest in the intermediate Theorem \ref{cocovergrothintro} comes from its applicability to Gro\-then\-dieck categories that originate as ``non-commutative deformations'' of schemes, more precisely abelian deformations of categories of quasi-coherent sheaves in the sense of \cite{lowenvandenbergh1}. 
After formulating a general result for deformations (Theorem \ref{thmdefcomp}), based upon lifting compact generators under deformation, we specialize further to the scheme case in Theorem \ref{thmscheme}. We use the description of deformations from \cite{lowen4} using non-commutative twisted presheaf deformations of the structure sheaf on an affine open cover.

When all involved deformed rings are commutative, using liftability of Koszul complexes under deformation, the corresponding twisted deformations are seen to be compactly generated, a fact which also follows from \cite{toen2}. 
In our main Theorem \ref{mainintro} (see Theorem \ref{maintheorem} in the paper), we show that actually all non-commutative deformations are compactly generated.

\begin{theorem} \label{mainintro} Let $X$ be a quasi-compact separated
  $k$-scheme. Then every flat deformation of the abelian category $\Qch(X)$  over a finite dimensional commutative local $k$-algebra has a compactly generated derived category.
\end{theorem}

The proof is based upon the following lifting result for Koszul complexes (see Theorem \ref{thmlift} in the paper):

\begin{theorem}\label{thmliftintro}
  Let $A$ be a commutative $k$-algebra and $f = (f_1, \dots, f_n)$ a
  sequence of elements in $A$. For $d\ge 1$ there exists a perfect
  complex $X_d \in D(A)$ generating the same thick subcategory of $D(A)$
as the Koszul complex $K(f)$ and satisfying the following property:  let $(R,m)$ be a finite dimensional commutative local $k$-algebra with $m^d = 0$ and $R/m
  = k$. Then $X_d$ may be lifted to any $R$-deformation of~$A$.
\end{theorem}

Concretely, let $\mathfrak{n}$  be the
Lie algebra freely generated by $x_1, \dots, x_n$ subject to the relations that all expressions involving $\geq d$ brackets vanish. Sending $x_i$ to~$f_i$
makes~$A$ into a right $\mathfrak{n}$-representation. Then $X_d$ is defined as the
Chevalley-Eilenberg complex $(A\otimes_k \wedge\mathfrak{n},d_{CE})$ of~$A$.
Clearly $X_1=K(f)$. 
It appears that in general one should think of $X_d$ as a kind of ``higher Koszul complex''.

\vspace{0,5cm}

\noindent \emph{Acknowledgement.}
The first author thanks Tobias Dyckerhoff, Dmitry Kaledin, Bernhard Keller, Alexander Kuznetsov, Daniel Murfet, Amnon Neeman, Paul Smith, Greg Stevenson and Bertrand To\"en for interesting discussions on the topic of this paper. She is especially grateful to Bernhard Keller for his valuable comments on a previous version of this paper.

\section{Coverings of Grothendieck categories}

Localization theory of abelian categories and Grothendieck categories goes back to the work of Gabriel \cite{gabriel}, which actually contains some of the important seeds of non-commutative algebraic geometry, like the fact that noetherian (this condition was later eliminated in the work of Rosenberg \cite{rosenberg2}) schemes can be reconstructed from their abelian category of quasi-coherent sheaves. In the general philosophy (due to Artin, Tate, Stafford, Van den Bergh and others) that non-commutative schemes can be represented by Grothendieck categories ``resembling'' quasi-coherent sheaf categories, localizations of such categories have been a key ingredient in the development of the subject by Rosenberg, Smith, Van Oystaeyen, Verschoren, and others (see eg. \cite{rosenberg}, \cite{smith}, \cite{vanoystaeyen}, \cite{vanoystaeyenverschoren}).
In particular, Van Oystaeyen and Verschoren investigated a notion of compatibility between localizations (see \cite{vanoystaeyen}, \cite{verschoren1}).

More recent approaches to non-commutative algebraic geometry (due to Bondal, Kontsevich, To\"en and others) take triangulated categories (and algebraic enhancements like dg or $A_{\infty}$ algebras and categories) as models for non-commutative spaces. The beautiful abelian localization theory was paralleled by an equally beautiful triangulated localization theory, based upon Verdier and Bousfield localization, see e.g. \cite{krause5} and the references therein. By considering appropriate unbounded derived categories, every Grothendieck localization gives rise to a Bousfield localization.

Recently, the notion of properly intersecting Bousfield subcategories was introduced by Rouquier in the context of his cocovering theorem concerning compact generation of certain triangulated categories \cite{rouquier}. The condition bears a striking similarity to the notion of compatibility in the Grothendieck context, which is even reinforced by the characterization proved by Murfet in \cite{murfet2}.

In this section, we introduce all the relevant notions in both contexts, and we observe that in the special situation where the right adjoints of a collection of compatible localizations of Grothendieck categories are exact, they give rise to properly intersecting Bousfield localizations, and Grothendieck coverings give rise to triangulated coverings.
We go on to deduce a covering theorem for Grothendieck categories (Theorem \ref{cocovergroth}) which allows to prove compact generation of the derived category.

\subsection{Coverings of abelian categories} \label{parcoverab}

We first review the situation for abelian categories.
Let $\ccc$ be an abelian category. A \emph{localization} of $\ccc$ consists of an exact functor
$$a: \ccc \lra \ccc'$$
with a fully faithful right adjoint $i: \ccc' \lra \ccc$. 
A subcategory $\sss \subseteq \ccc$ is called a \emph{Serre subcategory} if it is closed under subquotients and extensions. A Serre subcategory gives rise to an exact \emph{Gabriel quotient} 
$$a: \ccc \lra \ccc/\sss$$
with $\Kern(a) = \sss$. The Serre subcategory $\sss$ is called \emph{localizing} if $a$ is the left adjoint in a localization.  

Now suppose $\ccc$ is Grothendieck. Then $\sss$ is localizing precisely when $\sss$ is moreover closed under coproducts.
Conversely, for every localization $a: \ccc \lra \ccc'$, $\sss = \Kern(a)$ is localizing,
$a$ factors over an equivalence $\ccc/ \sss \cong \ccc'$, and putting
$$\sss^{\perp} = \{ C \in \ccc \,\, |\,\, \Hom_{\ccc}(S,C) = 0 = \Ext^1_{\ccc}(S,C)\,\, \forall S \in \sss \},$$
the right adjoint $i: \ccc' \lra \ccc$ factors over an equivalence $\ccc' \cong \sss^{\perp}$.

Let $\ccc$ be an abelian category.
For full subcategories $\sss_1$, $\sss_2$ of $\ccc$, the \emph{Gabriel product} is given by
$$\sss_1 \ast \sss_2 = \{ C \in \ccc \,\, |\,\, \exists\,\, S_1 \in \sss_1, \,\, S_2 \in \sss_2, \,\, 0 \lra S_1 \lra C \lra S_2 \lra 0\}.$$
Clearly, $\sss$ is closed under extensions if and only if $\sss \ast \sss = \sss$.
An easy diagram argument reveals that the Gabriel product is associative.

\begin{definition} \cite{vanoystaeyen}, \cite{verschoren1}
Full subcategories $\sss_1$, $\sss_2$ of $\ccc$ are called \emph{compatible} if
$$\sss_1 \ast \sss_2 = \sss_2 \ast \sss_1.$$
\end{definition}

For two compatible Serre subcategories, we have $\sss_1 \ast \sss_2 = \langle \sss_1 \cup \sss_2\rangle$, the smallest Serre subcategory containing $\sss_1$ and $\sss_2$.

Clearly, in the picture of a localization, the data of $\sss$, $a$ and $i$ determine each other uniquely.

\begin{proposition} \cite{vanoystaeyen} \cite{vanoystaeyenverschoren}\label{propcomp0}
Consider localizations $(\sss_1, a_1, i_1)$ and $(\sss_2, a_2, i_2)$ of $\ccc$. Put $q_1 = i_1 a_1$ and $q_2 = i_2 a_2$. The following are equivalent:
\begin{enumerate}
\item $\sss_1$ and $\sss_2$ are compatible.
\item $q_1(\sss_2) \subseteq \sss_2$ and $q_2(\sss_1) \subseteq \sss_1$.
\item $q_1 q_2 = q_2 q_1$.
\end{enumerate}
\end{proposition}

In the situation of Proposition \ref{propcomp0}, we speak about \emph{compatible} localizations.
A collection of Serre subcategories (or localizations) is called \emph{compatible} if the corresponding localizations are pairwise compatible.

\begin{definition}
A collection $\Sigma$ of Serre subcategories of $\ccc$ is called a \emph{covering} of $\ccc$ if
$$\bigcap \Sigma = \bigcap_{\sss \in \Sigma} \sss = 0.$$
\end{definition}

By this definition, the collection of functors $a: \ccc \lra \ccc/ \sss$ with $\sss \in \Sigma$ ``generates'' $\ccc$ in the sense that $C \in \ccc$ is non-zero if and only if $a(C)$ is non-zero for some $\sss \in \Sigma$.

\begin{proposition}
Consider a collection $\Sigma$ of localizing Serre subcategories of $\ccc$. The following are equivalent:
\begin{enumerate}
\item $\Sigma$ is a covering of $\ccc$.
\item The objects $i(D)$ for $D \in \ccc/\sss$ and $\sss \in \Sigma$ cogenerate $\ccc$, i.e a morphism $C' \lra C$ in $\ccc$ is non-zero if and only if there exists a morphism $C \lra i(D)$ with $D \in \ccc/\sss$ for some $\sss \in \Sigma$ such that  $C' \lra C \lra i(D)$ is non-zero.
\end{enumerate}
\end{proposition}

\begin{proof}
This easily follows from the adjunction between $a$ and $i$.
\end{proof}

The notion of a compatible covering is inspired by open coverings of schemes. For a covering collection $j: U_i \lra X$ of open subschemes of a scheme $X$ (i.e. $X = \cup U_i$), the collection of localizations $j^{\ast}: \Qch(X) \lra \Qch(U_i)$ consitutes a compatible covering of $\Qch(X)$.

\subsection{Descent categories}\label{pardescent}

Consider a compatible collection $\Sigma$ of localizations $\ccc_j$ of $\ccc$ indexed by a finite set $I$. We then obtain commutative (up to natural isomorphism) diagrams of localizations
$$\xymatrix{{\ccc} \ar[r]^{a_k} \ar[d]_{a_j} & {\ccc_k} \ar[d]^{a^k_{kj}} \\ {\ccc_j} \ar[r]_{a^j_{kj}} & {\ccc_{kj}}}$$
with $\ccc_k = \sss^{\perp}_k$, $\ccc_{kj} = (\sss_k \ast \sss_j)^{\perp} = \ccc_j \cap \ccc_k$. 
 
Using associativity of the Gabriel product and compatibility of the localizations, we obtain for each $J = \{ j_1, \dots, j_p\} \subseteq I$ a localizing subcategory
$$\sss_J = \sss_{j_1} \ast \dots \ast \sss_{j_p}$$
of $\ccc$ with corresponding localization
$$\ccc_J = \ccc_{j_1} \cap \dots \cap \ccc_{j_p}$$
of $\ccc$, and all the $\sss_J$ are compatible. In particular, we obtain for every inclusion $K \subseteq J$ a further localization $a^K_J: \ccc_K \lra \ccc_J$ left adjoint to the inclusion $i^K_J: \ccc_J \lra \ccc_K$. It is easily seen that for $K \subseteq J_1$ and $K \subseteq J_2$, the localizations $a^K_{J_1}$ and $a^K_{J_2}$ are compatible. Let $\Delta_{\varnothing}$ be the category of finite subsets of $I$ ordered by inclusions, and let $\Delta$ be the subcategory of non-empty subsets. Putting $\ccc_{\varnothing} = \ccc$, the categories $\ccc_J$ for $J \subseteq I$ can be organized into a pseudofunctor
$$\ccc_{\bullet}: \Delta_{\varnothing} \lra \Cat: J \longmapsto \ccc_J$$
with the $a^K_J$ for $K \subseteq J$ as restriction functors.
Hence, $\ccc_{\bullet}$ is a fibered category of localizations in the sense of \cite[Definition 2.4]{lowen4}.

We define the  \emph{descent category} $\Des(\Sigma)$ of $\Sigma$ to be the descent category $\Des(\ccc_{\bullet}|_{\Delta})$, i.e it is a bi-limit of the restricted pseudofunctor $\ccc_{\bullet}|_{\Delta}$.

In particular, we obtain a natural comparison functor
$$\ccc \lra \Des(\Sigma).$$
We conclude:

\begin{proposition}\label{propcovdes}
The compatible collection $\Sigma$ of localizations of $\ccc$ is a covering if and only if the comparison functor $\ccc \lra \Des(\Sigma)$ is faithful.
\end{proposition}

Conversely, descent categories yield a natural way of constructing an abelian category covered by a given collection of abelian categories. More precisely, let $I$ be a finite index set, let $\Delta$ be as above, and let 
$$\ccc_{\bullet}: \Delta \lra \Cat: J \lra \ccc_J$$
be a pseudofunctor for which every $a^K_J: \ccc_K \lra \ccc_J$ for $K \subseteq J$ is a localization with right adjoint $i^K_J$ and $\Kern(a^K_J) = \sss^K_J$. Suppose moreover that for $K \subseteq J_1$ and $K \subseteq J_2$ the corresponding localizations are compatible and $\sss^K_{J_1 \cup J_2} = \sss^K_{J_1} \ast \sss^K_{J_2}$.
Consider the descent category $\Des(\ccc_{\bullet})$ with canonical functors
$$a_K: \Des(\ccc_{\bullet}) \lra \ccc_K: (X_J)_J \longmapsto X_K$$

\begin{proposition}\label{desaffine}
\begin{enumerate}
\item The functor $a_K$ is a localization with fully faithful right adjoint $i_K$ with
$$a_J i_K = i^{J}_{K \cup J} a^K_{K \cup J}: \ccc_K \lra \ccc_J.$$
\item The localizations $a_K$ are compatible.
\item The localizations $a_K$ consitute a covering of $\Des(\ccc_{\bullet})$.
\item If the functors $i^{J}_{K \cup J}$ are exact, then so are the functors $i_K$.
\end{enumerate}
\end{proposition}

\begin{proof}
(1) For $X \in \ccc_K$, using compatibility of the localizations occuring in $\ccc_{\bullet}$, $(i^{J}_{K \cup J} a^K_{K \cup J}(X))_J$ can be made into a descent datum $i_K(X)$, and $i_K$ can be made into a functor right adjoint to $a_K$. Since $a_K i_K = 1_{\ccc_K}$, the functor $i_K$ is fully faithful. (2) 
(3) Immediate from Proposition \ref{propcovdes}. (4) Immediate from the formula in (1).
\end{proof}

\subsection{Coverings of triangulated categories}\label{parcovertria}

Next we review the situation for triangulated categories. For an excellent introduction to the localization theory of triangulated catgeories, we refer the reader to \cite{krause5}.

Let $\ttt$ be a triangulated category. A \emph{(Bousfield) localization} of $\ttt$ consists of an exact functor
$$a: \ttt \lra \ttt'$$
with a fully faithful (automatically exact) right adjoint $i: \ttt' \lra \ttt$. A subcategory $\iii \subseteq \ttt$ is called \emph{triangulated} if it is closed under cones and shifts and \emph{thick} if it is moreover closed under direct summands. A thick subcategory gives rise to an exact \emph{Verdier quotient}
$$a: \ttt \lra \ttt/\iii$$
with $\Kern(a) = \iii$. The thick subcategory $\iii$ is called a \emph{Bousfield subcategory} if $a$ is the left adjoint in a localization.

For every localization $a: \ttt \lra \ttt'$, $\iii = \Kern(a)$ is Bousfield,
$a$ factors over an equivalence $\ttt/ \iii \cong \ttt'$, and putting
$$\iii^{\perp} = \{ T \in \ttt \,\, |\,\, \Hom_{\ttt}(I,T) = 0  \,\, \forall I \in \iii \},$$
the right adjoint $i: \ttt' \lra \ttt$ factors over an equivalence $\ttt' \cong \iii^{\perp}$.

For full subcategories $\iii_1$, $\iii_2$ of $\ttt$, the \emph{Verdier product} is given by
$$\iii_1 \ast \iii_2 = \{ T \in \ttt \,\, |\,\, \exists\,\, I_1 \in \iii_1, \,\, I_2 \in \iii_2, \,\, I_1 \lra T \lra I_2 \lra \}.$$
Clearly, $\iii$ is triangulated if and only if $\iii = \iii \ast \iii$.

\begin{definition} \cite{rouquier} \cite{murfet2}
Full subcategories $\iii_1$, $\iii_2$ of $\ttt$ are said to \emph{intersect properly} if
$$\iii_1 \ast \iii_2 = \iii_2 \ast \iii_1.$$
\end{definition}

For two properly intersecting thick subcategories, $\iii_1 \ast \iii_2 = \langle \iii_1 \cup \iii_2\rangle$, the smallest thick subcategory containing $\iii_1$ and $\iii_2$.

Clearly, in the picture of a localization, the data of $\iii$, $a$ and $i$ determine eachother uniquely.

\begin{proposition}  \cite{rouquier} \cite{murfet2} \label{propcomp}
Consider localizations $(\iii_1, a_1, i_1)$ and $(\iii_2, a_2, i_2)$ of $\ccc$. Put $q_1 = i_1 a_1$ and $q_2 = i_2 a_2$. The following are equivalent:
\begin{enumerate}
\item $\iii_1$ and $\iii_2$ are compatible.
\item $q_1(\iii_2) \subseteq \iii_2$ and $q_2(\iii_1) \subseteq \iii_1$.
\item $q_1 q_2 = q_2 q_1$.
\end{enumerate}
\end{proposition}

In the situation of Proposition \ref{propcomp}, we speak about \emph{properly intersecting} localizations.
A collection of thick subcategories (or localizations) is called \emph{properly intersecting} if the localizations are pairwise properly intersecting.

\begin{definition}\label{defcovertria}
A collection $\Theta$ of full subcategories of $\ttt$ is called a \emph{covering} of $\ttt$ if
$$\bigcap \Theta = \bigcap_{\iii \in \Theta} \iii
 = 0.$$
\end{definition}

\begin{remark}
In \cite{rouquier}, the term cocovering is reserved for a collection of Bousfield subcategories which is covering in the sense of Definition \ref{defcovertria} and properly intersecting.
\end{remark}

By Definition \ref{defcovertria}, for a covering collection $\Theta$ of thick subcatories, the collection of quotient functors $a: \ttt \lra \ttt/ \iii$ with $\iii \in \Theta$ ``generates'' $\ttt$ in the sense that $T \in \ttt$ is non-zero if and only if $a(T)$ is non-zero for some $\iii \in \Theta$.

\begin{proposition}
Consider a collection $\Theta$ of Bousfield subcategories of $\ttt$. The following are equivalent:
\begin{enumerate}
\item $\Theta$ is a covering of $\ttt$.
\item The objects $i(D)$ for $D \in \ttt/\iii$ and $\iii \in \Theta$ cogenerate $\ttt$, i.e an object $T$ in $\ttt$ is non-zero if and only if there exists a non-zero morphism $T \lra i(D)$ with $D \in \ttt/\iii$ for some $\iii \in \Theta$.
\end{enumerate}
\end{proposition}

\begin{proof}
This easily follows from the adjunction between $a$ and $i$.
\end{proof}

\subsection{Induced coverings}

Given the formal parallellism between sections \ref{parcoverab} and \ref{parcovertria}, and the fact that a localization 
$a: \ccc \lra \ddd$
with right adjoint
$i: \ddd \lra \ccc$
of Grothendieck categories
gives rise to an induced
Bousfield localization
$$La = a: D(\ccc) \lra D(\ddd)$$
with fully faithful right adjoint $$Ri: D(\ddd) \lra D(\ccc)$$
of the corresponding derived categories, it is natural to ask what happens to the notions of compatibility and coverings under this operation of taking derived categories.

For coverings, the situation is very simple.
For a full subcategory $\sss$ of $\ccc$, let $D_{\sss}(\ccc)$ denote the full subcategory of $D(\ccc)$ consisting of complexes whose cohomology lies in $\sss$. 

\begin{lemma}
Let $a: \ccc \lra \ddd$ be an exact functor between Grothendieck categories with $\Kern(a) = \sss$, and consider $La = a: D(\ccc) \lra D(\ddd)$. We have $\Kern(La) = D_{\sss}(\ccc)$.
\end{lemma}

\begin{proof}
For a complex $X \in D(\ccc)$, we have $H^n(a(X)) = a(H^n(X))$.
\end{proof}

\begin{lemma}
For a collection $\Sigma$ of full subcategories of a Grothendieck category $\ccc$, we have
$$D_{\bigcap_{\sss \in \Sigma}\sss}(\ccc) = \bigcap_{\sss \in \Sigma} D_{\sss}(\ccc).$$
\end{lemma}

\begin{proposition}\label{propindcover}
Let $\Sigma$ be a collection of full subcategories of a Grothendieck category $\ccc$. Then $\Sigma$ is a covering of $\ccc$ if and only if the collection $\{ D_{\sss}(\ccc) \,\, |\,\, \sss \in \Sigma\}$ is a covering of $D(\ccc)$.
\end{proposition}

Now consider a Grothendieck category $\ccc$ and  localizations
$a_k: \ccc \lra \ddd_k$
with right adjoints $i_k$, $q_k = i_k a_k$, and $\sss_k = \Kern(a_k)$ for $k \in \{ 1, 2\}$. 

Taking derived functors yields Bousfield localizations $La_k = a_k: D(\ccc) \lra D(\ccc/\sss_k)$
with right adjoints $Ri_k: D(\ccc/\sss_k) \lra D(\ccc)$ and $\Kern(La_k) = D_{\sss_k}(\ccc)$.

We have the following inclusion between thick subcategories:
\begin{lemma}
We have $D_{\sss_1}(\ccc) \ast D_{\sss_2}(\ccc) \subseteq D_{\sss_1 \ast \sss_2}(\ccc)$.
\end{lemma}

\begin{proof}
A triangle $X_1 \lra X \lra X_2 \lra$ with $X_k \in D_{\sss_k}(\ccc)$ gives rise to a long exact sequence $\dots \lra H^nX_1 \lra H^nX \lra H^nX_2 \lra \dots$. Since $\sss_k$ is closed under subquotients, we obtain an exact sequence $0 \lra S_1 \lra H^nX \lra S_2 \lra 0$ with $S_k \in \sss_k$.
\end{proof}

In general, we have:

\begin{proposition}
If $D_{\sss_1}(\ccc)$ and $D_{\sss_2}(\ccc)$ are compatible in $D(\ccc)$, then $\sss_1$ and $\sss_2$ are compatible in $\ccc$.
\end{proposition}

\begin{proof}
Immediate from Lemma \ref{lemkey0} and the characterizations (2) in Propositions \ref{propcomp0} and \ref{propcomp}. \end{proof}

\begin{lemma} \label{lemkey0}
If $Ri_1 a_1(D_{\sss_2}(\ccc)) \subseteq D_{\sss_2}(\ccc)$, then $i_1 a_1(\sss_2) \subseteq \sss_2$.
\end{lemma}

\begin{proof}
Let $S_2 \in \sss_2$. We have $i_1 a_1 (S_2) = R^0i_1 a_1 (S_2) = H^0 Ri_1 a_1 (S_2)$ and since $Ri_1 a_1 (S_2) \in D_{\sss_2}(\ccc)$, it follows that $i_1 a_1 (S_2) \in \sss_2$ as desired.
\end{proof}

Unfortunately, the converse implication does not hold in general. However, in the special situation where $i_1$ and $i_2$ are exact, it is equally straightforward.

\begin{definition}
A localization $a: \ccc \lra \ddd$ is called \emph{affine} if the right adjoint $i: \ddd \lra \ccc$ is exact.  A localizing subcategory $\sss \subseteq \ccc$ is called \emph{affine} if the corresponding localization $\ccc \lra \ccc/\sss$ is affine.
\end{definition}

\begin{example}
If $X$ is a quasi-compact scheme and $U \subset X$ an affine open subscheme, then $\Qch(X) \lra \Qch(U)$ is an affine localization. 
\end{example}

\begin{remark}
Other variants of ``affineness'' have been used in the non-commutative algebraic geometry literature. For instance, Paul Smith calls an inclusion functor $i: \ddd \lra \ccc$ affine if it has both adjoints. In particular, if $i: \ddd \lra \ccc$ is the right adjoint in a localization, it becomes a forteriori exact. In this paper, we have chosen the most convenient  notion of ``affineness'' for our purposes.
\end{remark}

\begin{proposition}\label{compint}
If $\sss_1$ and $\sss_2$ are compatible and affine, then $D_{\sss_1}(\ccc)$ and $D_{\sss_2}(\ccc)$ are compatible. 
\end{proposition}

\begin{proof}
Immediate from Lemma \ref{lemkey} and the characterizations (2) in Propositions \ref{propcomp0} and \ref{propcomp}.
\end{proof}

\begin{lemma}\label{lemkey}
If $i_1a_1(\sss_2) \subseteq \sss_2$ and $i_1$ is exact, then $Ri_1 a_1(D_{\sss_2}(\ccc)) \subseteq D_{\sss_2}(\ccc)$.
\end{lemma}

\begin{proof}
For a complex $X \in D_{\sss_2}(\ccc)$, we have $Ri_1 a_1(X) = i_1 a_1 (X)$ and since $i_1 a_1$ is exact, $H^n(i_1 a_1 (X)) = i_1 a_1 (H^n(X)) \in \sss_2$.
\end{proof}

The affineness condition in Proposition \ref{compint} does not describe the only situation where compatible abelian localizations give rise to compatible Bousfield localizations, but it is the only situation we will need in this paper. To end this section, we will describe another situation, inspired by the behaviour of large categories of sheaves of modules.

Recall that the functor $i : \ccc/\sss \lra \ccc$ has finite cohomological dimension if there exists an $N \in \Z$ such that 
if $X \in D(\ccc/\sss)$ has $H^n(X) = 0$ for $n > 0$, then $R^n i(X) = H^n(Ri (X)) = 0$ for $n \geq N$.

\begin{proposition}\label{propfin}
Let $\sss_1$ and $\sss_2$ be compatible and suppose the following conditions hold:
\begin{enumerate}
\item There exist a class of objects $\aaa \subseteq \ccc$ and classes $\aaa_k \subseteq \ccc/ \sss_k$ consisting of $i_k$-acyclic objects such that $a_k(\aaa) \subseteq \aaa_k$ and $i_k(\aaa_k) \subseteq \aaa$.
\item The functors $i_k$ have finite cohomological dimension.
\end{enumerate}
Then $D_{\sss_1}(\ccc)$ and $D_{\sss_2}(\ccc)$ are compatible.
\end{proposition}

\begin{example}
Let $X$ be a quasi-compact scheme with quasi-compact open subschemes $j_1: U_1 \lra X$ and $j_2: U_2 \lra X$. We have restriction functors $j_k^{\ast}: \Mod(X) \lra \Mod(U_k)$ between the categories of all sheaves of modules with right adjoints $i_{k, \ast}: \Mod(U_k) \lra \Mod(X)$ with finite cohomological dimension. In Proposition \ref{propfin}, we can take for $\aaa$ and $\aaa_k$ the classes of flabby sheaves. Hence the localizations $j^{\ast}_k: D(\Mod(X)) \lra D(\Mod(U_k))$ are compatible. 

\end{example}

\subsection{Rouquier's Theorem}
Compactly generated triangulated categories were invented by Neeman \cite{neeman2} with the compact generation of derived categories of ``nice'' schemes as one of the principal motivations. 
As proved in \cite{neeman2}, for a scheme with a collection of ample invertible sheaves, these sheaves constitute a collection of compact generators of the derived category. But also in \cite{neeman2}, a totally different proof of compact generation is given for arbitrary quasi-compact separated schemes. The result is further improved by Bondal and Van den Bergh in \cite{bondalvandenbergh}, where a single compact generator is constructed for quasi-compact semi-separated schemes. These proofs are by induction on the opens in a finite affine cover, and the ingredients eventually cristalized in Rouquier's theorem \cite{rouquier} which is entirely expressed in terms of a cover of a triangulated category. Finally, in \cite{murfet2}, Murfet  obtained a version of the theorem with compactness replaced by $\alpha$-compactness. We start by recalling the theorem.

\begin{theorem}\cite{murfet2}\label{cocovertria}
Let $\ttt$ be a triangulated category with a compatible covering of Bousfield subcategories $\iii_i \subseteq \ttt$ for $i \in I = \{1, \dots, n\}$. Let $\alpha$ be a regular cardinal. Suppose:
\begin{enumerate}
\item $\ttt/ \iii_i$ is $\alpha$-compactly generated for every $i \in I$.
\item For every $i \in I$ and $\varnothing \neq J \subseteq I \setminus \{i\}$, the essential image of
$$\cap_{j \in J} \iii_j \lra \ttt \lra \ttt/ \iii_i$$
is $\alpha$-compactly generated in $\ttt/ \iii_i$.  
\end{enumerate}
Then $\ttt$ is $\alpha$-compactly generated, and an object in $\ttt$ is $\alpha$-compact if and only if its image in every $\ttt/ \iii_i$ for $i \in I$ is $\alpha$-compact.
\end{theorem}

\begin{remark}
The $\alpha = \aleph_0$-case of the theorem is Rouquier's cocovering theorem \cite{rouquier}.
\end{remark}

We now obtain the following application to Grothendieck categories:

\begin{theorem}\label{cocovergroth}
Let $\ccc$ be a Grothendieck category with a compatible covering of affine localizing subcategories $\sss_i \subseteq \ccc$ for $i \in I = \{1, \dots, n\}$. Suppose:
\begin{enumerate}
\item $D(\ccc/\sss_i)$ is $\alpha$-compactly generated for every $i \in I$.
\item For every $i \in I$ and $\varnothing \neq J \subseteq I \setminus \{i\}$, suppose the essential image $\eee$  of
$$\cap_{j \in J} \sss_j \lra \ccc \lra \ccc/ \sss_i$$
is such that $D_{\eee}(\ccc/\sss_i)$ is compactly generated in $D(\ccc/ \sss_i)$.
\end{enumerate}
Then $D(\ccc)$ is $\alpha$-compactly generated, and an object in $D(\ccc)$ is $\alpha$-compact if and only if its image in every $D(\ccc/ \sss_i)$ for $i \in I$ is $\alpha$-compact.
\end{theorem}

\begin{proof}
This is an application of Theorem \ref{cocovertria} by invoking Propositions \ref{propindcover} and \ref{compint} and Lemma \ref{lemess}.
\end{proof}

\begin{lemma}\label{lemess}
With the notations of Theorem \ref{cocovergroth}, $\cap_{j \in J} \sss_j$ is a localizing Serre subcategory which is compatible with $\sss_i$, and the essential image $\eee$ of
$$\cap_{j \in J} \sss_j \lra \ccc \lra \ccc/ \sss_i$$
is a localizing Serre subcategory given by the kernel of
$$\ccc/\sss_i \lra \ccc/(\sss_i \ast \cap_{j\in J} \sss_j).$$
The essential image of 
$$\cap_{j \in J}D_{\sss_j}(\ccc) \lra D(\ccc) \lra D(\ccc/\sss_i)$$
is given by $D_{\eee}(\ccc/\sss_i)$.
\end{lemma}

\begin{remark}
By the Gabriel-Popescu theorem, all Grothendieck categories are localizations of module categories, and thus their derived categories are well-generated \cite{neeman} \cite{krause6} (and thus $\alpha$-compactly generated for some $\alpha$) as localizations of compactly generated derived categories of rings. However, they are not necessarily compactly generated as was shown in \cite{neeman3}.
\end{remark}

\begin{remark}
Compatibility between localizations can be considered a commutative phenomenon (after all, it expresses that two localization functors commute). The non-commutative topology developed by Van Oystaeyen \cite{vanoystaeyen} encompasses notions of coverings (and in fact, non-commutative Grothendieck topologies) which apply to the situation of non-commuting localizations. An investigation whether this approach can be extended to the triangulated setup, and whether it is possible to obtain results on compact generation extending Theorems \ref{cocovertria} and \ref{cocovergroth}, is work in progress.
\end{remark}

\section{Deformations}

In this section we obtain an application of Theorem \ref{cocovergroth} to deformations of Grothendieck categories, based upon application to the undeformed categories (Theorem \ref{thmdefcomp}). For simplicity, we focuss on compact generation ($\alpha = \aleph_0$).
By the work of Keller \cite{keller1}, compact generation of the derived category $D(\ccc)$ of a Grothendieck category leads to the existence of a dg algebra $A$ - the derived endomorphism algebra of a generator - representing the category in the sense that $D(\ccc) \cong D(A)$. At this point, most of non-commutative derived algebraic geometry has been developed with dg algebras (or $A_{\infty}$-algebras) as models, although a definitive theory should also include more general algebraic enhancements on the level of the entire categories. For the topic of deformations, a satisfactory treatment on the level of dg algebras does certainly not exist in complete generality \cite{kellerlowen}, due to obstructions which also play an important role in the present paper. 
A deformation theory for triangulated categories on the level of enhancements of the entire categories is still under construction \cite{dedekenlowen2}, and is also subject to obstructions. Thus, Grothendieck enhancements are the only ones for which a satisfactory intrinsic deformation theory exists for the moment, and for this reason our intermediate Theorem \ref{cocovergroth} is crucial.

\subsection{Deformation and localization}

Infinitesimal deformations of abelian categories were introduced in \cite{lowenvandenbergh1}.
We deform along a surjective ringmap $R \lra k$ between coherent commutative rings, with a nilpotent kernel and such that $k$ is finitely presented over $R$. This includes the classical infinitesimal deformation setup in the direction of Artin local $k$-algebras.
Deformations are required to be flat in an appropriate sense, which was introduced in \cite{lowenvandenbergh1}. It was shown in the same paper that deformations of Grothendieck categories remain Grothendieck. The interaction between deformation and localization was treated in \cite[\S 7]{lowenvandenbergh1}.

Let $\iota: \ccc \lra \ddd$ be a deformation of Grothendieck categories. There are inverse bijections between the Serre subcategories of $\ccc$ and the Serre subcategories of $\ddd$ described by the maps
$$\sss \longmapsto \bar{\sss} = \langle \sss \rangle_{\ddd} = \{ D \in \ddd \,\, |\,\, k \otimes_R D \in \sss\}$$
and
$$\sss \longmapsto \sss \cap \ccc.$$
These restrict to bijections between localizing subcategories, and for corresponding localizing subcategories $\sss$ of $\ccc$ and $\bar{\sss}$ of $\ddd$, there is an induced deformation
$\ccc/\sss \lra \ddd/ \bar{\sss}$ and there are commutative diagrams
$$\xymatrix{ {\ddd} \ar[r]^{\bar{a}} & {\ddd/\bar{\sss}} \\ {\ccc} \ar[u] \ar[r]_a & {\ccc/ \sss} \ar[u]} \hspace{2cm}
\xymatrix{ {\ddd} & {\ddd/\bar{\sss}} \ar[l]_{\bar{i}} \\ {\ccc} \ar[u]  & {\ccc/ \sss.} \ar[u] \ar[l]^i}$$

\begin{proposition}\label{liftcover}
Let $\Sigma$ be a collection of Serre subcategories of $\ccc$ and consider the corresponding collection $\bar{\Sigma} = \{ \bar{\sss} \,\, |\,\, \sss \in \Sigma\}$ of Serre subcategories of $\ddd$.
Then $\Sigma$ is a covering of $\ccc$ if and only if $\bar{\Sigma}$ is a covering of $\ddd$.
\end{proposition}

\begin{proof}
Immediate from Lemma \ref{lemcovdef}.
\end{proof}

\begin{lemma}\label{lemcovdef}
Let $\Sigma$ be a collection of Serre subcategories of $\ccc$. We have 
$$\overline{ \cap_{\sss \in \Sigma} \sss} = \cap_{\sss \in \Sigma} \overline{\sss}.$$
\end{lemma}

\begin{proof}
Immediate from the description of the bijections between Serre subcategories of $\ccc$ and $\ddd$.
\end{proof}

\begin{proposition}\cite[Proposition 3.8]{lowen4}\label{liftcomp}
Let $\sss_k \subseteq \ccc$ be localizing subcategories for $k \in \{ 1,2\}$. If $\sss_1$ and $\sss_2$ are compatible in $\ccc$, then $\ovl{\sss_1}$ and $\ovl{\sss_2}$ are compatible in $\ddd$. In this case, we have $\ovl{\sss_1} \ast \ovl{\sss_2} = \overline{\sss_1 \ast \sss_2}$.
\end{proposition}

We will need one more lifting result.

\begin{lemma}\label{lemee}
Let $\sss_k \subseteq \ccc$ be compatible localizing subcategories for $k \in \{ 1,2\}$. The essential image $\eee$ of $$\sss_2 \lra \ccc \lra \ccc/\sss_1$$ is the kernel of 
$$a^1_2: \ccc/\sss_1 \lra \ccc/\sss_1 \ast \sss_2.$$
The lift $\ovl{\eee}$ of $\eee$ to $\ddd/\ovl{\sss_1}$ is the essential image of
$$\ovl{\sss_2} \lra \ddd \lra \ddd \lra \ddd/\ovl{\sss_2}.$$

\end{lemma}

\subsection{Lifts of compact generators}
Let $\iota: \ccc \lra \ddd$ be a deformation of Grothen\-dieck categories, let $\sss$ be a localizing Serre subcategory of $\ccc$ and let $\ovl{\sss}$ be the corresponding localizing subcategory of $\ddd$.

For an abelian category $\aaa$, let $\Ind(\aaa)$ be the ind-completion of $\aaa$, i.e. the closure of $\aaa$ inside $\Mod(\aaa)$ under filtered colimits, and let $\Pro(\aaa) = (\Ind(\aaa^{\op})^{\op}$ be the pro-completion of $\aaa$.

Consider the commutative diagram

$$\xymatrix{ {D_{\bar{\sss}}(\ddd)} \ar[r] & {D(\ddd)} \\ {D_{\sss}(\ccc)} \ar[u] \ar[r] & {D(\ccc)} \ar[u]_{R\iota} }$$
and the derived functor
$$k \otimes^L_R -: D(\mathsf{Pro}(\ddd)) \lra D(\mathsf{Pro}(\ccc)).$$

For a collection $\aaa$ of objects in a triangulated category $\ttt$, we denote by $\ovl{\langle \aaa \rangle }_\ttt$ the smallest localizing (i.e. triangulated and closed under direct sums) subcategory of $\ttt$  containing $\aaa$. 

Recall from \cite{dedekenlowen} that we have a balanced action
$$- \otimes^L_R - : D^{-}(\mmod(k)) \otimes D^{-}(\ddd) \lra D^{-}(\ddd).$$

The following is a refinement of \cite[Proposition 5.9]{dedekenlowen}.

\begin{proposition}\label{compgenlift}
Consider a collection $\GGGG$ of objects of $D^{-}(\ddd)$ such that the collection $k\otimes^L_R \GGGG = \{ k \otimes^L_R G \,\, |\,\, G \in \GGGG\}$ compactly generates $D_{{\sss}}(\ccc)$ inside $D(\ccc)$. Then $\GGGG$ compactly generates $D_{\ovl{\sss}}(\ddd)$ inside $D(\ddd)$.
\end{proposition}

\begin{proof}
The objects of $\GGGG$ are compact by \cite[Proposition 5.8]{dedekenlowen}. 
Consider $\ovl{ \langle \GGGG \rangle}_{D(\ddd)}$, i.e the closure of $\GGGG$ in $D(\ddd)$ under cones, shifts and direct dums.
We are to show that $\ovl{ \langle \GGGG \rangle}_{D(\ddd)} = D_{\ovl{\sss}}(\ddd)$.
First we have to make sure that every $G \in \GGGG$ is contained in $D_{\ovl{\sss}}(\ddd)$. Writing $I$ as a homotopy colimit of cones of finite free $k$-modules, we obtain that both $k \otimes^L_R G$ and $I \otimes^L_R G \cong I \otimes^L_k (k \otimes^L_R D)$ belong to $D_{\sss}(\ccc)$ and hence to $D_{\ovl{\sss}}(\ddd)$. From the triangle 
$$I \otimes^L_R G \lra G \lra k \otimes^L_R G \lra$$
we deduce that $G$ also belongs to $D_{\ovl{\sss}}(\ddd)$. Consequently $\ovl{ \langle \GGGG \rangle}_{D(\ddd)} \subseteq D_{\ovl{\sss}}(\ddd)$.

Next we look at the other inclusion $D_{\ovl{\sss}}(\ddd) \subseteq \ovl{ \langle \GGGG \rangle}_{D(\ddd)}$. 
For an arbitrary complex $D \in D_{\ovl{\sss}}(\ddd)$, we can write $D = \mathrm{hocolim}_{n = 0}^{\infty} \tau^{\leq n}D$ with $\tau^{\leq n}D \in D^{-}_{\ovl{\sss}}(\ddd)$. Consequently, it suffices to show that $D^{-}_{\ovl{\sss}}(\ddd) \subseteq \ovl{ \langle \GGGG \rangle}_{D(\ddd)}$. For $D \in D^-_{\ovl{\sss}}(\ddd)$, consider the triangle
$$I \otimes^L_R D \lra D \lra k \otimes^L_R D \lra.$$
First note that writing $k$ as a homotopy colimit of cones of finite free $R$-modules, we see that 
\begin{equation}\label{hocok}
k \otimes^L_R D \in \ovl{ \langle D \rangle}_{D(\ddd)}. 
\end{equation}
Using $I \otimes^L_R D \cong I \otimes^L_k (k \otimes^L_R D)$, we deduce from balancedness of the derived tensor product that $I \otimes_R^L D$ and $k \otimes_R^L D$ belong to both $D(\ccc)$ and $D_{\ovl{\sss}}(\ddd)$, whence to $D_{\sss}(\ccc)$. Consequently, it suffices to show that  $\ovl{ \langle k \otimes^L_R\GGGG \rangle}_{D(\ccc)} = D_{\sss}(\ccc) \subseteq \ovl{ \langle \GGGG \rangle}_{D(\ddd)}$. To see this, it suffices to consider \eqref{hocok} for all $D \in \GGGG$.
\end{proof}

\subsection{Compact generation of deformations}
Putting together all our results so far, we now describe a situation in which one obtains compact generation of the derived category $D(\ddd)$ of a deformation $\ddd$. 
Let $\ccc$ be a Grothendieck abelian category with a deformation $\iota: \ccc \lra \ddd$. Let $\sss_i \subseteq \ccc$ for $i \in I = \{1, \dots, n\}$ be a covering collection of compatible localizing subcategories of $\ccc$ and let $\ovl{\sss_i} \subseteq \ddd$ be the corresponding covering collection of compatible localizing subcategories of $\ddd$.

\begin{theorem}\label{thmdefcomp}
Suppose:
\begin{enumerate}
\item For every $i \in I$, $\ovl{\sss_i} \subseteq \ddd$ is affine.
\item For every $i \in I$, there is a collection $\GGGG_i$ of objects of $D^-(\ddd/\ovl{\sss_i})$ such that the collection $k \otimes_R^L \GGGG_i$ compactly generates $D(\ccc/\sss_i)$.
\item For every $i \in I$ and $J \subseteq I \setminus \{i\}$, the essential image $\eee$ of
$$\cap_{j \in J}\sss_j \lra \ccc \lra \ccc/\sss_i$$
is such that there is a collection $\GGGG$ of objects of $D^-(\ddd/\ovl{\sss_i})$ for which the collection $k \otimes_R^L \GGGG$ compactly generates $D_{{\eee}}(\ccc/{\sss_i})$ inside $D(\ccc/\sss_i)$.
\end{enumerate}
Then $D(\ddd)$ is compactly generated and an object in $D(\ddd)$ is compact if and only if its image in each $D(\ddd/ \ovl{\sss_i})$ is compact.
\end{theorem}

\begin{proof}
By Propositions \ref{liftcover} and \ref{liftcomp} and assumption (1), we are in the basic setup of Theorem \ref{cocovergroth}. By Proposition \ref{compgenlift}, the collections $\GGGG_i$ and $\GGGG$ in assumptions (2) and (3) consitute collections of compact generators of $D(\ddd/\sss_i)$ and of $D_{\ovl{\eee}}(\ddd/ \ovl{\sss}_i)$ inside $D(\ddd/ \ovl{\sss}_i)$ respectively. Finally, using Lemma \ref{lemee}, assumptions (1) and (2) in Theorem \ref{cocovergroth} are fulfilled and the theorem applies to the deformed situation.
\end{proof}

\section{Lifting Koszul complexes}

For an object $M$ in $D(A)$ for a $k$-algebra $A$ , we denote by
$\langle M \rangle_A$ the smallest thick subcategory of $D(A)$
containing $M$ and by $\ovl{\langle M \rangle}_A$ the smallest
localizing subcategory of $D(A)$ containing $M$. 

For a $k$-algebra map $A\lra B$ we have a restriction functor
\[
{}_A(-):D(B)\lra D(A)
\]
with left adjoint given by the derived tensor product
\[
B\otimes^L_A-:D(A)\lra D(B)
\]
(the actual map $A\lra B$ will never be in doubt).
\subsection{An auxiliary result}\label{paraux}

Let $k$ be a field and let $\NNN$ be the Lie algebra freely generated
by $x_1, \dots, x_n$ subject to the relations that all expressions
involving $\geq d$ brackets vanish. Since there are only a finite
number of expressions in $(x_i)_i$ involving $< d$ brackets, $\NNN$ is
finite dimensional over $k$. Let $U$ be the universal enveloping
algebra of $\NNN$. Let $I \subseteq U$ be the twosided ideal generated
by $([x_i, x_j])_{ij}$. Then $U/I = k[x_1, \dots, x_n]$. Our arguments
below will be mostly based on the $k$-algebra maps
\[
\xymatrix{%
U\ar[r] & U/I \ar[r]^-{x_i\mapsto 0}& k
}
\]

The left $k$-module $k$ gives rise to a left $U$-module ${_U}k = U/(x_1, \dots, x_n)$ and a left
$U/I$-module ${_{U/I}}k = (U/I)/(x_1, \dots, x_n)$. Since
$U$ (as well as $U/I$) is noetherian of finite global dimension,
${_U}k$ is a perfect left $U$-module and ${_{U/I}}k$ is a perfect
$U/I$-module. By tensoring we obtain
another perfect $U/I$-module: $U/I \otimes^L_U {_U}k$.

\begin{proposition}\label{propaux}
We have the following equalities: $\langle {_{U/I}}k \rangle_{U/I} = \langle U/I \otimes^L_U {_U}k \rangle_{U/I}$ and $\ovl{\langle {_{U/I}}k \rangle}_{U/I} = \ovl{\langle U/I \otimes^L_U {_U}k \rangle}_{U/I}$.
\end{proposition}

\begin{proof}
  We start by proving $ U/I \otimes^L_U {_U}k
  \in \langle {_{U/I}}k \rangle_{U/I}$. Note that
  $\langle {_{U/I}}k \rangle_{U/I}$ consists of all $U/I$-modules with
  finite dimensional total cohomology. Thus, it suffices to look at
  the cohomology of $U/I \otimes^L_U {_U}k$. Since we are only
  interested in the underlying $k$-module, it suffices to compute ${_k}(U/I
  \otimes^L_U {_U}k)$ (restriction for $k\hookrightarrow U/I$) 
which we can do using a finite free resolution of $(U/I)_U$ as
  a right $U$-module. Things then reduce to the trivial fact ${}_k(U\otimes^L_U {}_Uk)
\cong k$.

Thus, we have proven  $U/I \otimes^L_U {_U}k \in
  \langle {_{U/I}}k \rangle_{U/I} \subset \ovl{\langle {_{U/I}}k \rangle}_{U/I}$.
Now $U/I \otimes^L_U {_U}k$ is a perfect left $U/I$-module and
hence it is compact in $D(U/I)$. From this it follows immediately that
it is also a compact object in $\ovl{\langle {}_{U/I}k\rangle}_{U/I}$.
To prove the claims of the proposition it is sufficient to prove that
 $U/I \otimes^L_U {_U}k$
is a compact generator of $\ovl{\langle {}_{U/I}k\rangle}_{U/I}$. In other
words we have to prove that its right orthogonal is zero:
\[
\bigl( U/I
    \otimes^L_U {_U}k\bigr)^\perp\cap  \ovl{\langle {_{U/I}}k \rangle}_{U/I}  = 0.
\]

Now suppose we have $X \in
  \ovl{\langle {_{U/I}}k \rangle}_{U/I}$ with $\RHom_{U/I}(U/I
  \otimes^L_U {_U}k, X) = 0$. Then we have ${}_UX \in \ovl{\langle
    {_U}k \rangle}_U$ and also by adjunction $\RHom_{U}({_U}k, {}_UX) =
  0$. Since the perfect complex ${}_Uk$ is a compact generator of $\ovl{\langle
    {_U}k \rangle}_U$ we obtain ${}_UX=0$ which implies $X=0$.
\end{proof}

\subsection{Koszul precomplexes}\label{koszul}

Let $A$ be a possibly non-commutative $k$-algebra and consider a finite sequence of elements $x = (x_1, \dots, x_n)$ in $A$. 
We will work in the category $\Mod(A)$ of left $A$-modules. 
We define a precomplex $K(x)$ of $A$-modules with 
$K(x)_p = A\otimes_k \Lambda^p k^n$ the free $A$-module of rank $\begin{pmatrix} n \\ p \end{pmatrix}$ with basis $e_{i_1} \wedge \dots \wedge e_{i_p}$ with $i_1 < \dots < i_p$.
We define the $A$-linear morphism $$d_p: K(x)_p \lra K(x)_{p-1}$$
by $$d_p(e_{i_1} \wedge \dots \wedge e_{i_p}) = \sum_{k = 1}^p (-1)^{k+1} x_{i_k} e_{i_1} \wedge \dots \wedge \hat{e}_{i_k} \wedge \dots \wedge e_{i_p}.$$

The differential may be compactly written as $d=\sum_iR_{x_i}\partial/\partial e_i$
where we consider the $e_i$ as odd and $R_{x_i}(a)=ax_i$ which yields:
\[
d^2=\sum_{1 \leq i < j \leq p} R_{[x_j,x_i]}\,\partial^2/\partial e_i\partial e_j
\]
Thus $K(x)$ is a complex if and only if the $(x_{i})_i$  commute.

\subsection{Lifting Koszul complexes}\label{parliftkoszul}
Let $(R, m)$ be a finite dimensional $k$-algebra with $m^d = 0$ and
$R/m = k$ and let $A'$ be an $R$-algebra with $A'/mA' = A$. Consider a
sequence $f = (f_1, \dots, f_n)$ of element in $A$ and a sequence $f'
= (f'_1, \dots, f'_n)$ of elements in $A'$ such that the reduction of
$f'_i$ to $A$ equals $f_i$.  Let $K(f)$ be the Koszul complex
associated to $f$. As soon as some of the $f'_i$ do not commute, the
Koszul precomplex $K(f')$ fails to be a complex according to \S
\ref{koszul}. For this reason, we will now use the result of \S
\ref{paraux} to lift a perfect complex generating the same localizing
subcategory as $K(f)$. In fact, this ``liftable complex'' happens to
be independent of $A'$ or $R$! Its size depends however in a major way on
$d$.

\begin{theorem}\label{thmlift}
Let $(R,m)$ be a finite dimensional algebra with $m^d = 0$ and $R/m = k$, and let $A$ be a commutative $k$-algebra and $f = (f_1, \dots, f_n)$ a sequence of elements in $A$.
There exists a perfect complex $X \in D(A)$ with $\langle K(f) \rangle_A = \langle X \rangle_A$ and
$$\ovl{\langle K(f) \rangle}_A = \ovl{\langle X \rangle}_A$$
which is such that for every $R$-algebra $A'$ with $A'/m = A$ there exists a perfect complex $X' \in D(A')$ with $A \otimes^L_{A'} X' = X$.
We can take $X = A \otimes^L_U {_U}k$.
\end{theorem}

\begin{proof}
Let $f' = (f'_1, \dots, f'_n)$ be an arbitrary sequence of elements in $A'$ such that the reduction of $f'_i$ to $A$ equals $f_i$. 
From the definition of the algebra $U$ in \S \ref{paraux}, we obtain a commutative diagram
$$\xymatrix{{U} \ar[r] \ar[d] & {A'} \ar[d] \\ {U/I} \ar[r] & A}$$
with the horizontal maps determined by $x_i \longmapsto f'_i$ and $x_i \longmapsto f_i$ respectively.
We thus have
\begin{equation}\label{eqdiag}
A \otimes^L_{A'} (A' \otimes^L_U {{_U} k}) = A \otimes^L_{U/I} (U/I \otimes^L_U {_U}k).
\end{equation}
By Proposition \ref{propaux} we have 
$\langle {_{U/I}}k \rangle_{U/I} = \langle U/I \otimes^L_U {_U}k \rangle_{U/I}$
and hence
\begin{equation}\label{eqsubcats}
\langle A \otimes^L_{U/I} {_{U/I}} k \rangle_{A} = \langle A \otimes^L_{U/I} (U/I \otimes^L_U {_U}k)\rangle_A.
\end{equation}
Over $U/I = k[x_1, \dots, x_n]$, the Koszul complex $K(x_1, \dots, x_n)$ constitutes a projective resolution of ${_{U/I}}k$. Hence, on the left hand side of \eqref{eqsubcats} we have $A \otimes^L_{U/I} {_{U/I}} k = A \otimes_{U/I} K(x_1, \dots, x_n) = K(f)$. Hence, by \eqref{eqdiag} it suffices to take $X = A \otimes^L_U {_U}k$ and $X' = A' \otimes^L_U {_U}k$.
\end{proof}

Since over $U$, the Chevalley-Eilenberg complex $V(\NNN)$ of $\NNN$ constitutes a projective resolution of ${_U}k$, in Theorem \ref{thmlift} we concretely obtain $X = A \otimes_U V(\NNN) = A \otimes_k \Lambda^{\ast}\NNN$ and $X' = A' \otimes_U V(\NNN) = A' \otimes_k \Lambda^{\ast}\NNN$, both equipped with the Chevalley-Eilenberg differential
$$\begin{aligned}
d(a \otimes y_1 \wedge \dots \wedge y_p) & = \sum_{i = 1}^p (-1)^{i +1} ay_i \otimes y_1 \wedge \dots \wedge \hat{y_i} \wedge \dots \wedge y_p \\
& + \sum_{i < j} (-1)^{i + j} a \otimes [y_i, y_j] \wedge \dots \wedge \hat{y_i} \dots \wedge \hat{y_j} \dots \wedge y_p
\end{aligned}$$
for a basis $(y_i)_i$ for $\frak{n}$.
Let us look at some examples. 

If $d = 1$ or $n = 1$, we have $\NNN = kx_1 \oplus \dots \oplus kx_n$, $U = k[x_1, \dots, x_n]$, $V(\NNN) = K(x_1, \dots, x_n)$ and $X = K(f_1, \dots, f_n)$. For $n = 1$ we have $X' = K(f'_1)$. 

Thus, the first non-trivial case to consider is $d = 2$ and $n = 2$. We have $\NNN = kx_1 \oplus kx_2 \oplus k[x_1, x_2]$ and  consequently $X$ is given by the complex
$$\xymatrix{0 \ar[r] & A \ar[r]_{d_3} & {A^3} \ar[r]_{d_2} & {A^3} \ar[r]_{d_1} & A \ar[r] & 0}$$
with basis elements over $A$ given by $x_1 \wedge x_2 \wedge [x_1, x_2]$ in degree 3, $x_1\wedge x_2$, $x_2 \wedge [x_1, x_2]$, $[x_1, x_2] \wedge x_1$ in degree 2, $x_1$, $x_2$, $[x_1, x_2]$ in degree 1 and $1$ in degree 0 and differentials given by
$$d_3 = \begin{pmatrix} [f_1, f_2] \\ f_1 \\ f_2 \end{pmatrix}, \hspace{0,5cm} d_2 = \begin{pmatrix} -f_2 & 0 & [f_1, f_2] \\
f_1 & -[f_1, f_2] & 0 \\ -1 & f_2 & - f_1 \end{pmatrix}, \hspace{0,5cm} d_1 = \begin{pmatrix} f_1 & f_2 & [f_1, f_2] \end{pmatrix}.$$
Similarly $X'$ is given by the same expressions with $A$ replaced by $A'$ and $f_i$ replaced by the chosen lift $f'_i$. Note that $[f_1, f_2] = 0$ but we possibly have $[f'_1, f'_2] \neq 0$.

\section{Deformations of schemes} \label{pardefschemes}

In this section we specialize Theorem \ref{thmdefcomp} to the scheme case. In Theorem \ref{thmscheme}, we give a general formulation in the the setup of a Grothendieck deformation of the category $\Qch(X)$ over a quasi-compact, separated scheme. After discussing some special cases in which direct lifting of Koszul complexes already leads to compact generation of the deformed category (like the case in which all deformed rings on an affine cover are commutative), in \S \ref{parmaintheorem} we prove our main Theorem \ref{maintheorem} which states that all non-commutative deformations are in fact compactly generated. The proof is based upon the change from Koszul complexes to liftable generators from Theorem \ref{thmlift}.

\subsection{Deformed schemes using ample line bundles}
Let $X$ be a quasi-compact separated scheme over a field $k$. If we want to investigate compact generation of $D(\ddd)$ for an abelian deformation $\ddd$ of $\ccc = \Qch(X)$, by Proposition \ref{compgenlift} (and in fact, its special case \cite[Proposition 5.9]{dedekenlowen}) a global approach is to look for compact generators of $D(\ccc)$ that lift to $D(\ddd)$ under $k \otimes^L_R -$. We easily obtain the following result:

\begin{proposition}\label{propamp}
Suppose $X$ has an ample line bundle $\LLL$. If $H^2(X, \ooo_X) = 0$, all infinitesimal deformations of $\Qch(X)$ have compactly generated derived categories.
\end{proposition}

\begin{proof}
According to \cite{neeman}, $D(\Qch(X))$ is compactly generated by the tensor powers $\LLL^n$ for $n \in \Z$. By \cite{lowen2}, the obstructions to lifting $\LLL^n$ along an infinitesimal deformation lie in $\Ext^2_{X}(\LLL^n, I \otimes_k \LLL^n)$ for $I \cong k^m$ for some $m$. But we have
$$\Ext^2_{X}(\LLL, I \otimes_k \LLL) \cong [\Ext^2(\LLL, \LLL)]^m = [\Ext^2_X(\ooo_X, \ooo_X)]^m = [H^2(X, \ooo_X)]^m = 0$$
as desired.
\end{proof}

\subsection{Deformed schemes using coverings}\label{pardefschcover}

Let $(X, \ooo)$ be a quasi-compact, separated scheme and put $\ccc = \Qch(X)$. Since the homological condition in Proposition \ref{propamp} excludes interesting schemes, we now investigate a different approach based upon affine covers.

Let $U_i$ for $i \in I = \{1, \dots, n\}$ be an affine cover of $X$, with $U_i \cong \mathrm{Spec}(\ooo(U_i))$. Put $Z_i = X \setminus U_i$. With $\ccc_i = \Qch(U_i)$ and $\sss_i = \Qch_{Z_i}(X)$, the category of quasi-coherent sheaves on $X$ supported on $Z_i$, we are in the situation of a covering collection of compatible localizations of $\ccc$. 

For $J \subseteq I$, put $U_J = \cap_{j \in J}U_j$ and $\ccc_J = \Qch(U_J)$. For $i \in I$ and $J \subseteq I \setminus \{i\}$, put $Z^i_J = U_i \setminus \cup_{j \in J} U_j = U_i \cap \cap_{j \in J} Z_j$. The essential image $\eee$ of $\cap_{j \in J}\sss_j \lra \ccc \lra \ccc/\sss_i$ is given by $\eee = \Qch_{Z^i_J}(U_i)$. 

Let $\Delta$ and $\Delta_{\varnothing}$ be as in \S \ref{pardescent}. For $K \subseteq J$, we have $U_J \subseteq U_K$ and the corresponding localization is given by restriction of sheaves $a^K_J: \Qch(U_K) \lra \Qch(U_J)$ with right adjoint direct image functor $i^K_J$. Moreover, the localization can be entirely described in terms of module categories. If $\ooo(U_K) \lra \ooo(U_J)$ is the canonical restriction, then we have 
$$a^K_J \cong \ooo(U_J) \otimes_{\ooo(U_K)} -: \Mod(\ooo(U_K)) \lra \Mod(\ooo(U_J))$$
and the right adjoint $i^K_J$ is simply the restriction of scalars functor, which is obviously exact.
For the resulting pseudofunctor
$$\Mod(\ooo(U_{\bullet})): \Delta \lra \Cat: J \lra \Mod(\ooo(U_{J})),$$
we have $\Qch(X) \cong \Des(\Mod(\ooo(U_{\bullet})))$.
According to \cite{lowen4}, this situation is preserved under deformation. More precisely, up to equivalence an arbitrary abelian deformation $\iota: \ccc \lra \ddd$ is obtained as
$\ddd \cong \Des(\Mod(\ovl{\ooo}_{\bullet}))$ where 
$$\ovl{\ooo}_{\bullet}: \Delta \lra \Rng: J \lra \ovl{\ooo}_J$$
is a pseudofunctor (a ``twisted presheaf'') deforming $\ooo(U_{\bullet})$, and $\ddd_J = \Mod(\ovl{\ooo}_J)$ is the deformation of $\ccc_J$ corresponding to $\ovl{\sss_J}$. By Proposition \ref{desaffine}, the functors
$\ovl{i}_K: \ddd_K \lra \ddd$ are exact, and the $\ovl{a}_k: \ddd \lra \ddd_k$ constitute a covering collection of compatible localizations of $\ddd$. We note that by taking $\GGGG_i = \{\ovl{\ooo}_i\}$, condition 2 in Theorem \ref{thmdefcomp} is automatically fulfilled. We conclude:

\begin{theorem}\label{thmscheme}
Let $X$ be a quasi-compact, separated scheme with an affine cover $U_i$ for $i \in I = \{1, \dots, n\}$. Let $\iota: \Qch(X) \lra \ddd$ be an abelian deformation with induced deformations $\ddd_i$ of $\Qch(U_i)$. For every $i \in I$ and $J \subseteq I \setminus \{i\}$, consider $Z^i_J = U_i \cap \cap_{j \in J} Z_j$. Suppose there is a collection $\GGGG^i_J$ of objects in $D^-(\ddd_i)$ such that $k \otimes^L_R \GGGG^i_J$ compactly generates $D_{Z^i_J}(U_i)$ inside $D(U_i)$.
Then $D(\ddd)$ is compactly generated and an object in $D(\ddd)$ is compact if and only if its image in each $D(\ddd_i)$ is compact.
\end{theorem}

\begin{remark}
Before it makes sense to investigate the more general situation of deformations of quasi-compact, semi-separated schemes $X$, for which $D_{\Qch(X)}(\Mod(X))$ is known to be compactly generated by \cite{bondalvandenbergh}, a better understanding of the direct relation between deformations of $\Qch(X)$ and $\Mod(X)$ should be obtained. It follows from \cite{lowen4} that these two Grothendieck categories have equivalent deformation theories, the deformation equivalence passing through twisted non-commutative deformations of the structure sheaf. An interesting question in its own right is to understand whether corresponding deformations of $\Qch(X)$ and of $\Mod(X)$ are related by an inclusion functor and a quasi-coherator like in the undeformed setup. \end{remark}

\subsection{Twisted deformed schemes}\label{partwisted}

In this section we collect some observations which follow immediately from Theorem \ref{thmscheme}, based upon direct lifting of Koszul complexes. In the slightly more restrictive deformation setup of \S \ref{parmaintheorem}, all compact generation results we state here also follow from the more general Theorem \ref{maintheorem}, but there the involved generators are more complicated.

Let $A$ be a commutative $k$-algebra and $f = (f_1, \dots, f_n)$ a finite sequence of elements in $A$. Put $X = \mathrm{Spec}(A)$.
Consider the closed subset
$$Z = V(f) = V(f_1, \dots, f_n) = \{ p \in \mathrm{Spec}(A) \,\, |\,\, f_1, \dots, f_n \in p\} \subseteq X.$$
Let $\Qch_Z(X)$ be the localizing subcategory of quasi-coherent sheaves on $X$ supported on $Z$ and put $D_Z(X) = D_{\Qch_Z(X)}(\Qch(X))$.
We recall the following:

\begin{proposition}\cite{bokstedtneeman}\label{propbok}
The category $D_Z(X)$ is compactly generated by $K(f)$ inside $D(X)$.
\end{proposition}

Let $\ovl{A}$ be an $R$-deformation of $A$ and let $\ddd = \Mod(\ovl{A})$ be the corresponding abelian deformation of $\ccc = \Qch(X) \cong \Mod(A)$. Let $\ovl{\Qch_Z(X)} \subseteq \ddd$ be the localizing subcategory corresponding to $\Qch_Z(X) \subseteq \ccc$.
Let $\ovl{f} = (\ovl{f_1}, \dots, \ovl{f_n})$ be a sequence of lifts of the elements $f_i \in A$ to $\ovl{f_i}$ in $\ovl{A}$ under the canonical map $\ovl{A} \lra A$.
Clearly the precomplex $K(\ovl{f})$ of finite free $\ovl{A}$-modules
satisfies $k \otimes_R K(\ovl{f}) = K(f)$. 
If $K(\ovl{f})$ is a complex, we thus have $$k \otimes^L_R K(\ovl{f}) = K(f).$$
From Proposition \ref{compgenlift} we deduce:

\begin{proposition}
If every two distinct elements $\ovl{f_l}$ and $\ovl{f_k}$ commute, then $K(\ovl{f})$ compactly generates $D_{\ovl{\Qch_Z(X)}}(\ddd)$ inside $D(\ddd)$.
\end{proposition}

\begin{corollary}
If $\ovl{A}$ is a commutative deformation of $A$, then $D_{\ovl{\Qch_Z(X)}}(\ddd)$ is compactly generated inside $D(\ddd)$.
\end{corollary}

\begin{corollary}\label{corsing}
If $f = (f_1)$ consists of a single element, then $D_{\ovl{\Qch_Z(X)}}(\ddd)$ is compactly generated inside $D(\ddd)$.
\end{corollary}

We can now formulate some corollaries of Theorem \ref{thmscheme}:

\begin{proposition}
Let $X$ be a quasi-compact separated scheme with affine cover $U_i$ for $i \in I = \{1, \dots, n\}$. 
Let
$$\ovl{\ooo}_{\bullet}: \Delta \lra \Rng: J \longmapsto \ovl{\ooo}_J$$
be a pseudofunctor deforming
$$\ooo(U_{\bullet}): \Delta \lra \Rng: J \longmapsto \ooo(U_{\bullet})$$
such that all the rings $\ovl{\ooo}_J$ are commutative.
Then the category $D(\ddd)$ for $\ddd = \Des(\Mod(\ovl{\ooo}_{\bullet}))$ is compactly generated and an object in $D(\ddd)$ is compact if and only if its image in each of the categories $\Mod(\ovl{\ooo}_i)$ is compact.
\end{proposition}

In particular, we recover the fact that for a smooth scheme, the components $H^2(X, \ooo_X) \oplus H^1(X, \ttt_X)$ of $HH^2(X)$ correspond to compactly generated deformations of $\Qch(X)$, a fact which also follows from \cite{toen2}.

\begin{proposition}\label{curves}
 Let $X$ be a scheme with an affine cover $U_1$, $U_2$ with $U_i \cong \mathrm{Spec}(A_i)$ such that $U_1 \cap U_2 \cong \Spec({A_1}_{x}) \cong \Spec({A_2}_y)$ for $x \in A_1$ and $y \in A_2$. Then every deformation of $\Qch(X)$ is  compactly generated.
\end{proposition}

Unfortunately, Proposition \ref{curves} typically applies to curves, and they tend to have no genuinely non-commutative deformations. For instance, for a smooth curve $X$ the Hochschild cohomology is seen to reduce to $HH^2(X) = H^1(X, \ttt_X)$ for dimensional reasons, whence there are only scheme deformations of $X$.

\subsection{Non-commutative deformed schemes}\label{parmaintheorem}

Let $k$ be a field and $(R,m)$ a finite dimensional $k$ algebra with $m^d = 0$ and $R/m = k$.
In this section we prove our main result, namely that non-commutative deformations of quasi-compact separated schemes are compactly generated. Based upon \S \ref{parliftkoszul}, we remedy the fact that for general non-commutative deformations of schemes, the relevant Koszul precomplexes fail to be complexes and hence cannot be used as lifts, unlike in the special cases discussed in \S \ref{partwisted}.

\begin{theorem}\label{maintheorem}
Let $X$ be a quasi-compact separated $k$-scheme with an affine cover $U_i$ for $i \in I = \{1, \dots, n\}$. Let $\iota: \Qch(X) \lra \ddd$ be an abelian $R$-deformation with induced deformations $\ddd_i$ of $\Qch(U_i)$. Then $D(\ddd)$ is compactly generated and an object in $D(\ddd)$ is compact if and only if its image in each $D(\ddd_i)$ is compact.
\end{theorem}

\begin{proof}
For $i \in I$ and $J \subseteq I \setminus \{i\}$, put $Y = U_i = \Spec(A)$ and $Z = U_i \cap \cap_{j \in J} Z_j$. For a finite sequence of elements $f = (f_1, \dots, f_k)$ we can write
$$Z = V(f) = \{ p \in \Spec(A) \,\, |\,\, f_1, \dots. f_k \in p \} \subseteq Y.$$
For the induced deformation $\ddd_i$ of $\Qch(U_i) \cong \Mod(A)$ we have $\ddd_i \cong \Mod(\ovl{A})$ for an $R$-deformation $\ovl{A}$ of $A$. 
By Proposition \ref{propbok}, the category $D_Z(Y)$ is compactly generated by $K(f)$ inside $D(Y)$. Now by Theorem \ref{thmlift}, there exists a perfect complex $X' \in D(\ovl{A})$ for which $A \otimes^L_{\ovl{A}} X' = k \otimes^L_R X'$ compactly generates $D_X(Y)$ inside $D(Y) \cong D(A)$ as desired.
\end{proof}

The theorem shows in particular that the entire second Hochschild cohomology is realized by means of compactly generated abelian deformations.

\section{Appendix: Removing obstructions}\label{parremobs}

In this appendix we discuss an approach to removing obstuctions to first order deformations from \cite{kellerlowen} based upon the Hochschild complex, which applies in the case of length two Koszul complexes and thus leads to an alternative proof of Theorem \ref{maintheorem} in the case of first order deformations of surfaces. We compare the explicit lifts we obtain in both approaches. 

\subsection{Hochschild complex}

Let $\PPP$ be a $k$-linear abelian category. Recall that the Hochschild complex $\CC(\PPP)$ is the complex of $k$ modules with, for $n \geq 0$,
$$\CC^n(\PPP) = \prod_{P_0, \dots, P_n \in \ccc} \Hom_k(\PPP(P_{n-1}, P_n) \otimes \dots \otimes \PPP(P_0, P_1), \PPP(P_0, P_n))$$
endowed with the familiar Hochschild differential.
Let $C(\PPP)$ be the dg category of complexes of $\PPP$-objects with Hochschild complex $\CC(C(\PPP))$ with
$$\CC^n(C(\PPP)) = \prod_{C_0, \dots, C_n \in C(\PPP)} \Hom_k(\Hom(C_{n-1}, C_n) \otimes \dots \otimes \Hom(C_0, C_1), \Hom(C_0, C_n)).$$
An element of $C^n(\PPP)$ can be naively extended to $C(\PPP)$, yielding
$$\CC^n(\PPP) \lra \CC^n(C(\PPP)): \phi \lra \phi.$$

\subsection{Linear deformations and lifts of complexes}

If $m$ is the compositon of the category $\PPP$, then a Hochschild $2$-cocycle $\phi$ corresponds to the first order deformation
$$(\ovl{\PPP} = \PPP[\epsilon], \ovl{m} = m + \phi \epsilon).$$
Here $\Ob(\ovl{\PPP}) \cong \Ob(\PPP)$ and we denote objects in $\ovl{\PPP}$ by $\ovl{P}$ for $P \in \PPP$.
For objects $P_0, P_1 \in \PPP$, we have $\ovl{\PPP}(\ovl{P_1}, \ovl{P_0}) = \PPP(P_1, P_0)[\epsilon]$. A morphism $f: P_1 \lra P_0$ in $\PPP$ naturally gives rise to a morphism $\ovl{f} = f + 0\epsilon: \ovl{P_1} \lra \ovl{P_0}$ in $\ovl{\PPP}(\ovl{P_1}, \ovl{P_0})$, the trivial lift.
For a complex $(P, d)$ of $\PPP$-objects, there thus arises a natural lifted precomplex $(\ovl{P}, \ovl{d})$ of $\ovl{\PPP}$-objects.

In $\ovl{\PPP}$ we have
$$\ovl{m}(\ovl{d}, \ovl{d}) = \phi(d,d) \epsilon$$
and in fact, $$[\phi(d,d)] \in K(\PPP)(P[-2],P)$$
is precisely the obstruction to the existence of a \emph{complex} $(\ovl{P}, \ovl{d}')$ in $K(\ovl{\PPP})$ with $k \otimes_R (\ovl{P}, \ovl{d}') \cong (P,d)$ in $K(\PPP)$ (see \cite{lowen2}).
In general this obstruction will not vanish, but in some cases it is seen to vanish on the nose.

\begin{proposition}
If the differential $d$ of $P$ has no two consecutive non-zero components $d_n: P_n \lra P_{n-1}$, then $\phi(d,d) = 0$ and $(\ovl{P}, \ovl{d})$ is a complex lifting $(P,d)$.
\end{proposition}

\subsection{Removing obstructions}\label{par2remobs}
If $0 \neq [\phi(d,d)] \in \Ext^2_{\PPP}(P,P)$, following \cite{kellerlowen} we consider the morphism $\phi(d,d): \Sigma^{-2} P \lra P$ and we turn to the related complex
$$P^{(1)} = \cone(\phi(d,d)) = P \oplus \Sigma^{-1} P$$
with differential
$$d^{(1)} = \begin{pmatrix} d & \phi(d,d) \\ 0 & -d \end{pmatrix}.$$
The obstruction associated to the complex $P^{(1)}$ is then given by
$$\phi^{(1)} = \phi(d^{(1)}, d^{(1)}) = \begin{pmatrix} \phi(d,d) & \phi(d, \phi(d,d)) - \phi(\phi(d,d), d) \\
0 & \phi(d,d) \end{pmatrix}.$$
According to \cite[Lemma 3.18]{kellerlowen}, the degree two morphism
$$\begin{pmatrix} \phi(d,d) & 0\\
0 & \phi(d,d) \end{pmatrix}: P^{(1)} \lra P^{(1)}$$
is nullhomotopic (a nullhomotopy is given by $\small{\begin{pmatrix} 0 & 0 \\ 1 & 0 \end{pmatrix}}$).
 
Put $\psi^{(1)} = \phi(d, \phi(d,d)) - \phi(\phi(d,d), d)$. It follows that the obstruction associated to $P^{(1)}$ is given by
$$\tilde{\phi}^{(1)} = \begin{pmatrix} 0 & \psi^{(1)} \\ 0 & 0 \end{pmatrix} \in K(\PPP)(P^{(1)}[-2], P^{(1)}).$$

In general there is no reason why $\tilde{\phi}^{(1)}$ should be nullhomotopic, but in some cases it can be seen to be zero on the nose.

\begin{proposition}\label{propconelift}
Suppose the differential $d$ of $P$ has no three consecutive non-zero components $d_n: P_n \lra P_{n-1}$. Then $\psi^{(1)} = 0$ and there is a complex $\ovl{P^{(1)}} \in K(\ovl{\PPP})$ with $k \otimes_R \ovl{P^{(1)}} \cong P^{(1)} \in K(\PPP)$.
\end{proposition}

Following \cite[Proposition 3.16]{kellerlowen}, we note that the original complex $P$ can sometimes be reconstructed from $P^{(1)}$.

\begin{proposition} \cite{kellerlowen} \label{propconegen}
If $\phi(d,d)$ is nilpotent, $P$ can be constructed from $P^{(1)}$ using cones, shifts and direct summands. This applies in particular if $d$ has no $m$ consecutive non-zero components $d_n: P_n \lra P_{n-1}$ for some $m \geq 1$.
\end{proposition}

\begin{proof}
This follows from the octahedral axiom (see \cite[Proposition 3.16]{kellerlowen}).
\end{proof}

\subsection{The case of Koszul complexes}
The approach to removing obstructions discussed in \S \ref{par2remobs} applies to the case of length two Koszul complexes. In this section we compare this approach with the solution from \S \ref{parliftkoszul}.
Let $A$ be a commutative $k$-algebra with a first order deformation $\bar{A}$ determined by a Hochschild 2 cocycle $\phi \in \Hom_k(A \otimes_k A, A)$. For a sequence $f = (f_1, f_2)$ of elements in $A$ we consider the Koszul complex $(K(f), d)$ which is given by
$$\xymatrix{0 \ar[r] & A \ar[r]_-{\small{\begin{pmatrix} f_1 \\ -f_2 \end{pmatrix}}} & {A^2} \ar[r]_{\small{\begin{pmatrix} f_2 & f_1 \end{pmatrix}}} & A \ar[r] & 0}$$
In \S \ref{parremobs}, we take $\PPP$ to be the category of finite free $A$-modules.
The obstruction $\phi(d,d): \Sigma^{-2}K(f) \lra K(f)$ is determined by the element
$$\alpha = \phi(f_1, f_2) - \phi(f_2, f_1) \in A.$$
The complex $K(f)^{(1)} = \cone(\phi(d,d))$ is given by
$$\xymatrix{0 \ar[r] & A \ar[r]_{d^{(1)}_3} & {A^3} \ar[r]_{d^{(1)}_2} & {A^3} \ar[r]_{d^{(1)}_1} & A \ar[r] & 0}$$
with differentials given by
$$d^{(1)}_3 = \begin{pmatrix} 0 \\ -f_1 \\ f_2 \end{pmatrix}, \hspace{0,5cm} d^{(1)}_2 = \begin{pmatrix} -f_2 & 0 & 0 \\
f_1 & 0 & 0 \\ \ \alpha & f_2 &  f_1 \end{pmatrix}, \hspace{0,5cm} d^{(1)}_1 = \begin{pmatrix} - f_1 & - f_2 & 0 \end{pmatrix}.$$
Apart from the signs, the main difference with the complex $X$ from \S \ref{parliftkoszul} lies in the fact that here $\alpha$ depends on the Hochschild cocycle, whereas in $X$ it is replaced by the constant value $1$.
The nulhomotopy $\partial$ for the obstruction $\phi^{(1)}$ gives rise to the lifted complex $K(f)^{(1)}[\epsilon]$ with differential $d - \partial \epsilon$. Concretely, the differential is given by
$$\bar{d}^{(1)}_3 = \begin{pmatrix} - \epsilon \\ -f_1 \\ f_2 \end{pmatrix}, \hspace{0,5cm} \bar{d}^{(1)}_2 = \begin{pmatrix} -f_2 & 0 & - \epsilon \\
f_1 & - \epsilon & 0 \\ \ \alpha & f_2 &  f_1 \end{pmatrix}, \hspace{0,5cm} \bar{d}^{(1)}_1 = \begin{pmatrix} - f_1 & - f_2 & - \epsilon \end{pmatrix}.$$

On the other hand, if for $X'$ we choose $f'_i = f_i + 0 \epsilon$, then we have $[f'_1, f'_2] = \alpha \epsilon$ and hence $X'$ has differential $d'$ given by
$$d'_3 = \begin{pmatrix} \alpha \epsilon \\ f_1 \\ f_2 \end{pmatrix}, \hspace{0,5cm} d'_2 = \begin{pmatrix} -f_2 & 0 & \alpha \epsilon \\
f_1 & -\alpha \epsilon & 0 \\ -1 & f_2 & - f_1 \end{pmatrix}, \hspace{0,5cm} d'_1 = \begin{pmatrix} f_1 & f_2 & \alpha \epsilon \end{pmatrix}.$$
Clearly, the computations leading to $(\bar{d}^{(1)})^2 = 0$ and to ${d'}^2 = 0$ are almost identical and have the definition of $\alpha$ as main ingredient.

\def\cprime{$'$} \def\cprime{$'$}
\providecommand{\bysame}{\leavevmode\hbox to3em{\hrulefill}\thinspace}
\providecommand{\MR}{\relax\ifhmode\unskip\space\fi MR }
\providecommand{\MRhref}[2]{%
  \href{http://www.ams.org/mathscinet-getitem?mr=#1}{#2}
}
\providecommand{\href}[2]{#2}

\end{document}